\documentclass{svproc}

\usepackage{epsfig}
\usepackage{epstopdf}
\usepackage{amsmath,amssymb}
\usepackage{subfig}

\begin{document}
	
\mainmatter  	
	
\title{A Stochastic Capital-Labour Model\\ with Logistic Growth Function\thanks{This is 
a preprint whose final form is published by Springer Nature Switzerland AG in the book 
'Dynamic Control and Optimization'.}}

\titlerunning{A Stochastic Capital-Labour Model}  

\author{Houssine Zine\inst{1} \and Jaouad Danane\inst{2} \and Delfim F. M. Torres\inst{3}}

\authorrunning{H. Zine, J. Danane and D. F. M. Torres} 

\tocauthor{Houssine Zine, Jaouad Danane and Delfim F. M. Torres}

\institute{Center for Research and Development in Mathematics and Applications (CIDMA),\\ 
Department of Mathematics, University of Aveiro, 3810-193 Aveiro, Portugal\\ 
\email{zinehoussine@ua.pt}; \texttt{http://orcid.org/0000-0002-8412-7783}
\and
Laboratory of Systems, Modelization and Analysis for Decision Support,\\
Hassan First University, National School of Applied Sciences, Berrechid, Morocco\\ 
\email{jaouaddanane@gmail.com}; \texttt{http://orcid.org/0000-0001-7080-9743}
\and
Center for Research and Development in Mathematics and Applications (CIDMA),\\ 
Department of Mathematics, University of Aveiro, 3810-193 Aveiro, Portugal\\ 
\email{delfim@ua.pt}; \texttt{http://orcid.org/0000-0001-8641-2505}}

\maketitle  


\begin{abstract}
We propose and study a stochastic capital-labour model with logistic 
growth function. First, we show that the model has a unique positive 
global solution. Then, using the Lyapunov analysis method, we obtain 
conditions for the extinction of the total labour force. Furthermore, 
we also prove sufficient conditions for their persistence in mean. 
Finally, we illustrate our theoretical results through numerical simulations.

\keywords{capital-labour mathematical model; 
stochastic differential equations;
Brownian motion;
extinction and persistence.}

\medskip

\noindent {\bf MSC 2020:} 91B70. 
\end{abstract}


\section{Introduction}

Labour supply and demand are essential variables governing the labour market.
They are influenced by demographic factors and the gross domestic product, which 
vary from household to household. In our context, the supply of labour 
is represented by the number of free jobs and the demand for labour, noting 
that the workforce or labour force is the total number of people eligible 
to work \cite{MR4154640,MR4068867}.

Motivated by the previous information, we propose to model the labour market 
by ordinary differential equations (ODEs) describing 
the different interactions between the essential components, that is, 
the free jobs and the labour force. The suggested
model will take the following form:
\begin{equation}
\label{sy1}
\left\{
\begin{aligned}
\dfrac{du}{dt}(t)&=ru(t)\left(1-\dfrac{u(t)}{K}\right)-mu(t)v(t),\\
\dfrac{dv}{dt}(t)&=mu(t)v(t)-dv(t),
\end{aligned}
\right.
\end{equation}
where $u$ denotes the number of free jobs and $v$ represents 
the total unemployed labour force. The positive constant $r$ 
is the natural per capita growth of free jobs and $K$ 
is the theoretical eventual maximum of the number
of free jobs (related to the theoretical maximum of investment capital). 
The positive parameter $d$ is the disappearance rate of labour force 
and $muv$ is the rate by which the labour force fills in the free jobs.
We have adopted the bilinear form to pass from the labour force compartment 
to the free job one, while the recruitment of people depends progressively 
and proportionally to the considered employment policy.

It is well known that economies are subject to randomness in terms
of natural perturbation processes \cite{MR3891564}. Therefore, stochastic 
models are more suitable than deterministic ones, because they can take into 
account not only the mean trend but also the variance structure around it.
Moreover, deterministic models will always produce the same results 
for fixed initial conditions, whereas the stochastic ones may give different predicted values.
Thus, in order to take into account all the previous arguments,
in this paper we propose the following stochastic capital-labour model 
with a logistic growth function: 
\begin{equation}
\label{sy2}
\left\{
\begin{aligned}
du(t)&=\left[ru(t)\left(1-\dfrac{u(t)}{K}\right)-mu(t)v(t)\right]dt-\sigma u(t)v(t)dB,\\
dv(t)&=\left[mu(t)v(t)-dv(t)\right]dt+\sigma u(t)v(t)dB,
\end{aligned}
\right.
\end{equation}
where $B(t)$ is a standard Brownian motion with intensity $\sigma$, 
defined on a complete filtered probability space 
$\left( \Omega,\mathcal{F},(\mathcal{F}_t)_{t\geq 0},\mathbb{P}\right)$ 
with the filtration $(\mathcal{F}_t)_{t\geq 0}$ satisfying the usual 
conditions \cite{MyID:456}. The motivation to use the logistic growth
can be found in \cite{Lotfi:2019}; the reader interested in the stochastic 
techniques is refereed to \cite{Mahrouf:Axioms,MR4173153}
and references therein.

Our work is organized as follows. First, in Section~\ref{sec:2},
we prove existence and uniqueness of a global positive solution 
to our stochastic model \eqref{sy2}. 
Then, using the Lyapunov analysis method, we prove in Section~\ref{sec:3} 
the extinction of the total labour force under an appropriate condition.
Furthermore, in Section~\ref{sec:4}, we give sufficient conditions 
for the persistence in mean of the total labour force.
Follows some numerical simulations to illustrate our analytical results 
(Section~\ref{sec:5}). We finish with Section~\ref{sec:6} of conclusions.

All the equations and inequalities in the paper 
are understood in the almost surely (a.s.) sense.


\section{Existence and uniqueness of global economic solutions}
\label{sec:2}

To investigate the dynamical behaviour of a population model, 
the first concern is whether the solution of the model is positive and global. 
In order to get a stochastic differential equation for which a unique global 
solution exists, i.e., there is no explosion within a finite time, for any initial value, 
standard assumptions for existence and uniqueness of solutions are
the linear growth condition and the local Lipschitz condition 
(cf. Mao \cite{14}). However, the coefficients 
of system \eqref{sy2} do not satisfy the linear growth condition 
as the incidence is non-linear. Therefore, the solution of system \eqref{sy2} 
may explode at a finite time. In this section, using the Lyapunov analysis method \cite{9,11}, 
we show that the solution of system \eqref{sy2} is positive and global.

\begin{theorem}
For any given initial value $(u(0),v(0)) \in \mathbb{R}^2_+$, 
there exists a unique positive solution $(u(t),v(t)) \in \mathbb{R}^2_+$ 
of model \eqref{sy2} for all $t\geq 0$ a.s. Moreover,
\begin{equation*}
\limsup_{t\rightarrow \infty} u(t)\leq \dfrac{rK}{\mu} \text{ a.s.},
\quad
\limsup_{t\rightarrow \infty} v(t) \leq \dfrac{rK}{\mu} \text{ a.s.},
\end{equation*} 
where $\mu =\min\{r,d\}$.
\end{theorem}

\begin{proof}
Since the drift and the diffusion of \eqref{sy2} are locally Lipschitz, 
then for any given initial value $(u(0),v(0))\in\mathbb{R}^2_+$, 
there exists a unique local solution for $t\in [0,\tau_e)$, 
where $\tau_e$ is the explosion time. To show that this solution is global,
we need to show that $\tau_e=+\infty$. Define the stopping time $\tau^+$ as
$$
\tau^+ := \inf\left\{t\in [0,\tau_e)\; 
: u(t)\leq 0 \text{ or }  v(t)\leq 0\right\}.
$$
We suppose that $\tau^+ < +\infty$. For any $t\leq \tau^+$, 
we define the following function: 
\begin{equation*}
F(t) := \ln(u(t)v(t)).
\end{equation*}
By using It\^o's formula and system \eqref{sy2}, we obtain that
\begin{align*}
dF&=r\left(1-\dfrac{u}{K}\right)-mv+mu-d-\dfrac{\sigma^2}{2}(u^2+v^2)+\sigma(u-v)dB\\
&\geq -\dfrac{r}{K}u-mv-d-\dfrac{\sigma^2}{2}(u^2+v^2)+\sigma(u-v)dB.
\end{align*}
Integrating both sides between $0$ and $t$, we get that
\begin{equation}
\label{GS1}
F(t) \geq F(0)+\int_0^t H(s)ds+\sigma \int_0^t(u(s)-v(s))dB(s),
\end{equation}
where $H(s)=-\dfrac{r}{K}u(s)-mv(s)-d-\dfrac{\sigma^2}{2}(u^2(s)+v^2(s))$.
At least one among $u(\tau^+)$ and $v(\tau^+)$ is equal to $0$.
Then, we get 
$$
\lim_{t\rightarrow\tau^+}F(t)=-\infty.
$$
Letting $t\rightarrow\tau^+$ in \eqref{GS1} we obtain
\begin{equation*}
-\infty \geq F(t)
\geq F(0)+\int_0^{\tau^+} H(s)ds
+\sigma \int_0^{\tau^+}(u(s)-v(s))dB(s)> -\infty,
\end{equation*}
which is a contradiction. Thereby, $\tau^+=+\infty$, which means that the model
has a unique global solution $(u(t), v(t)) \in \mathbb{R}^2_+$ a.s. 
We now prove the boundedness. If we sum the equations 
from system \eqref{sy2}, then
$$
dN(t)=\left( ru(t)\left(1-\dfrac{u(t)}{K}\right)-dv(t)\right) dt,
$$
where $N(t)=u(t)+v(t)$. Thus,
\begin{align*}
dN(t)&=\left( ru\left(1-\dfrac{u}{K}\right)-dv\right) dt\\
&=\left( ru-dv-\dfrac{r}{K}(u-K)^2-2ur+rK\right) dt\\
&=\left( -ru-dv-\dfrac{r}{K}(u-K)^2+rK\right) dt,\\
\dfrac{dN}{dt}&\leq  -\mu N +rk,
\end{align*}
where $\mu =\min\{r,d\}$, and so
\begin{align*}
e^{\mu t}\dfrac{dN}{dt}&\leq e^{\mu t}\left( -\mu N +rk\right),\\
\int_0^t e^{\mu s}\dfrac{dN}{ds}ds&\leq \int_0^t e^{\mu s}\left( -\mu N(s) +rK\right)ds,\\
e^{\mu t} N(t)&\leq \dfrac{rK}{\mu}(e^{\mu t}-1)+N(0),\\
 N(t)&\leq \dfrac{rK}{\mu}(1-e^{-\mu t})+N(0) e^{-\mu t},\\
\limsup_{t\rightarrow \infty} N(t)&\leq \dfrac{rK}{\mu} \text{ a.s.}
\end{align*}  
This fact implies that
$\limsup_{t\rightarrow \infty} u(t)\leq \dfrac{rK}{\mu} \text{ a.s.}$
and $\limsup_{t\rightarrow \infty} v(t)\leq \dfrac{rK}{\mu} \text{ a.s.}$,
which completes the proof.\qed
\end{proof} 


\section{Extinction of total labour force}
\label{sec:3}

When studying dynamical systems, it is important 
to discuss the possibility of extinction or 
persistence of a population. Here we investigate 
extinction of the capital-labour. The question
of persistence will be addressed in Section~\ref{sec:4}.

\begin{theorem}
\label{thm3.1}
For any initial data $(u(0),v(0))\in \mathbb{R}^2_+$,
if $ \dfrac{m^2}{2\sigma^2}-d<0$, then one has
$v(t)\rightarrow 0$  a.s. when $t\rightarrow +\infty$.
\end{theorem}

\begin{proof} 
Let us define $G_1(t):=\log(v(t))$. 
Applying It\^o's formula to $G_1$ leads to
\begin{align*}
dG_1(t)&=\left(mu(t)-d-\dfrac{\sigma^2}{2}u^2(t)\right)dt+\sigma u(t)dB(t),\\
dG_1(t)&=\left(-\dfrac{\sigma^2}{2}\left( u(t)-\dfrac{m}{\sigma^2}\right)^2 
+\dfrac{m^2}{2\sigma^2}-d\right)dt+\sigma u(t)dB(t),\\
dG_1(t)& \leq \left(\dfrac{m^2}{2\sigma^2}-d\right)dt+\sigma u(t)dB(t).
\end{align*}
Integrating from $0$ to $t$ and dividing both sides by $t$, we have
\begin{align*}
\dfrac{\log(v(t))}{t}
&\leq \dfrac{\log(v_0)}{t}+\dfrac{1}{t}
\int_0^t \left(\dfrac{m^2}{2\sigma^2}-d\right)ds
+\dfrac{\sigma}{t}\int_0^t u(s)dB(s)\\
&\leq \dfrac{\log(v_0)}{t}+\dfrac{1}{t}\int_0^t 
\left(\dfrac{m^2}{2\sigma^2}-d\right)ds
+\dfrac{\sigma}{t}\int_0^t u(s)dB(s)\\
&\leq \dfrac{\log(v_0)}{t}+\dfrac{m^2}{2\sigma^2}-d
+\dfrac{\sigma}{t}\int_0^t u(s)dB(s).
\end{align*}
Let $M_t:=\displaystyle \int_0^t \sigma u(s)dB_s$. Then,
\begin{align*}
\limsup_{t\rightarrow+\infty} \dfrac{<M_t,M_t>}{t}
=\limsup_{t\rightarrow+\infty} \dfrac{\sigma^2}{t}\int_0^t u^2(s)ds
\leq \sigma^2\left( \dfrac{rK}{\mu}\right)^2<+\infty
\end{align*}
and, by using the strong law of large numbers for martingales 
(see, e.g., \cite{14}),
\begin{align*}
\limsup_{t\rightarrow+\infty}\dfrac{M_t}{t}=0.
\end{align*}
Therefore,
\begin{equation*}
\limsup_{t\rightarrow+\infty}\dfrac{\log(v(t))}{t}
\leq  \dfrac{\log(v_0)}{t}+\dfrac{m^2}{2\sigma^2}-d.
\end{equation*}
Thus, if $ \dfrac{m^2}{2\sigma^2}-d<0$, 
then $v(t)\rightarrow 0$ when $t\rightarrow +\infty$ a.s.\qed
\end{proof}


\section{Persistence in the mean of total labour force}
\label{sec:4}

Now, we investigate the persistence property of $v(t)$ in the mean,
that is, we give conditions for which 
$$
\liminf_{t\rightarrow \infty}\dfrac{1}{t}\int_0^t v(s)ds>0.
$$
For convenience, we introduce the following notation:
$$
\langle x(t)\rangle := \dfrac{1}{t}\int_0^t x(s)ds.
$$

\begin{theorem}
\label{thm:4.1}
Let $(u(t), v(t))$ be a solution of system \eqref{sy2} 
with initial value 
$$
(u(0), v(0)) \in \mathbb{R}^2_+.
$$ 
If 
\begin{equation}
\label{eq:cond:pers}
R^s_0 := \dfrac{r}{d}-\dfrac{\sigma^2K^2}{2d}>1 
\quad \text{ and } \quad
m>\dfrac{r}{K},
\end{equation}
then the variable $v(t)$ satisfies the following expression:
\begin{equation}
\liminf_{t\rightarrow \infty} \langle v \rangle 
\geq  \dfrac{d(R^s_0-1)}{m+d}>0.
\end{equation}
\end{theorem}

\begin{proof}
Using the second equation of system \eqref{sy2}, we have
\begin{equation}
\label{vt}
\dfrac{v(t)-v(0)}{t}=m\langle u v \rangle
-d \langle v\rangle+\dfrac{\sigma}{t}\int_0^t  u(s)v(s)dB.
\end{equation}
Applying It\^o's formula on model \eqref{sy2} leads to
\begin{equation}
\label{lnu}
d\ln (u(t))=\left[r\left(1-\dfrac{u}{K}\right)-mv
-\dfrac{1}{2}\sigma^2v^2\right]dt-\sigma v dB
\end{equation}
and
\begin{equation}
\label{lnv}
d\ln (v(t))=\left[mu-d-\dfrac{1}{2}\sigma^2u^2\right]dt+\sigma u dB.
\end{equation}
Integrating both sides of \eqref{lnu} and \eqref{lnv} 
from $0$ to $t$, and dividing by $t$, leads to
\begin{equation}
\label{lnu1}
\dfrac{\ln(u(t))-\ln(u(0))}{t}=r-\dfrac{r}{K}\langle u\rangle-m\langle v\rangle
-\dfrac{\sigma^2}{2}\langle v^2\rangle-\dfrac{\sigma}{t}\int_0^t  v(s)dB
\end{equation}
and
\begin{equation}
\label{lnv1}
\dfrac{\ln(v(t))-\ln(v(0))}{t}=m\langle u\rangle-d
-\dfrac{\sigma^2}{2}\langle u^2\rangle+\dfrac{\sigma}{t}\int_0^t  u(s)dB.
\end{equation}
Combining \eqref{vt}, \eqref{lnu1}, and \eqref{lnv1}, we derive that
\begin{equation*}
\begin{split}
\dfrac{v(t)-v(0)}{t}
&+\dfrac{\ln(u(t))-\ln(u(0))}{t}+\dfrac{\ln(v(t))-\ln(v(0))}{t}\\
&=r-\dfrac{r}{K}\langle u\rangle+m\langle u\rangle+(m-\sigma^2)\langle u v\rangle
-(d+m)\langle v\rangle-d\\
&\quad -\dfrac{\sigma^2}{2}\langle (u+v)^2\rangle
+\dfrac{\sigma}{t}\int_0^t(u(s)-v(s)+u(s)v(s))dB_s\\
&\geq r-d-\dfrac{\sigma^2}{2}K^2- (d+m)\langle v\rangle+\left(m-\dfrac{r}{K}\right)
\langle u\rangle\\
&\quad +\dfrac{\sigma}{t}\int_0^t(u(s)-v(s)+u(s)v(s))dB_s.
\end{split}
\end{equation*}
Since $m-\dfrac{r}{K}>0$, then
\begin{align*}
\dfrac{v(t)-v(0)}{t}&+\dfrac{\ln(u(t))-\ln(u(0))}{t}+\dfrac{\ln(v(t))-\ln(v(0))}{t}\\
& \geq  r-d-\dfrac{\sigma^2}{2}K^2- (d+m)\langle v\rangle
+\dfrac{\sigma}{t}\int_0^t(u(s)-v(s)+u(s)v(s))dB_s.
\end{align*}
Therefore,
\begin{align*}
(d+m)\langle v\rangle\geq 
& r-d-\dfrac{\sigma^2}{2}K^2
+\dfrac{\sigma}{t}\int_0^t(u(s)-v(s)+u(s)v(s))dB(s)\\
&-\dfrac{v(t)-v(0)}{t}-\dfrac{\ln(u(t))-\ln(u(0))}{t}
-\dfrac{\ln(v(t))-\ln(v(0))}{t}.
\end{align*}
Let us denote
$$
M_1(t):=\sigma\int_0^t(u(s)-v(s)+u(s)v(s))dB(s).
$$
Using the strong law of large numbers for martingales,
together with the fact that almost surely for every 
$\varepsilon$ there exists $T$ such that 
$0 < u(t), v(t) < \frac{rK}{\mu} + \varepsilon$ 
for every $t > T$, we can say that
\begin{equation*}
\lim_{t\rightarrow \infty} \dfrac{v(t)}{t}=0,
\quad \lim_{t\rightarrow \infty} \dfrac{u(t)}{t}=0,
\quad \lim_{t\rightarrow \infty} \dfrac{M_1(t)}{t}=0
\quad a.s.
\end{equation*}
Thus,
\begin{align*}
\liminf_{t\rightarrow \infty} \langle v\rangle\geq 
& \dfrac{r-d-\dfrac{\sigma^2}{2}K^2}{(d+m)}=\dfrac{d(R^s_0-1)}{m+d}>0,
\end{align*}
where $R^s_0=\dfrac{r}{d}-\dfrac{\sigma^2K^2}{2d}$.
The proof is complete.\qed
\end{proof}


\section{Numerical simulations}
\label{sec:5}

In this section, we illustrate our mathematical results
through numerical simulations. In the two examples considered, 
we apply the algorithm presented in \cite{Higham} 
to solve system \eqref{sy2} and we use the parameter 
values from Table~\ref{tabl1}, inspired from \cite{Riad}.
\begin{table}
\caption{Parameter values used in the numerical simulations.}
\label{tabl1}
\centering
\begin{tabular}{|c|c|c|c|}\hline \hline
Parameters  & Fig.~\ref{fig1} & Fig.~\ref{fig2}\\ \hline \hline
$r$ & $1$&  $1$\\
$d$	&   $0.2$&$0.2$\\
$m$ &$0.001$&$0.1$\\
$K$   & $100$& $100$   \\
$\sigma$ & $0.09$& $0.001$\\ \hline
\end{tabular}
\end{table}

Figure~\ref{fig1} shows the evolution of the free jobs and the total 
labour force during the period of observation. It can be seen that 
both curves of the total labour force, corresponding to the deterministic 
and to the stochastic models, converge toward zero. This indicates the extinction 
of the total labour force, which is consistent with our theoretical results.
Indeed, for the used parameters (see Table~\ref{tabl1}), one has 
$\dfrac{m^2}{2\sigma^2}-d=-0.19<0$ and it follows, from
Theorem~\ref{thm3.1}, that $v(t)\rightarrow 0$ 
with probability one when $t\rightarrow +\infty$. 
\begin{figure}
\centering
\subfloat[Behaviour of free jobs]{\label{fig:1a}
\includegraphics[scale=0.23]{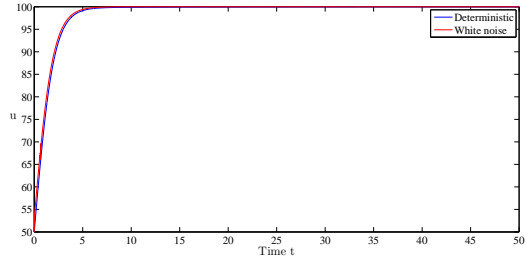}}\\
\subfloat[Behaviour of total labour force]{\label{fig:1b}
\includegraphics[scale=0.23]{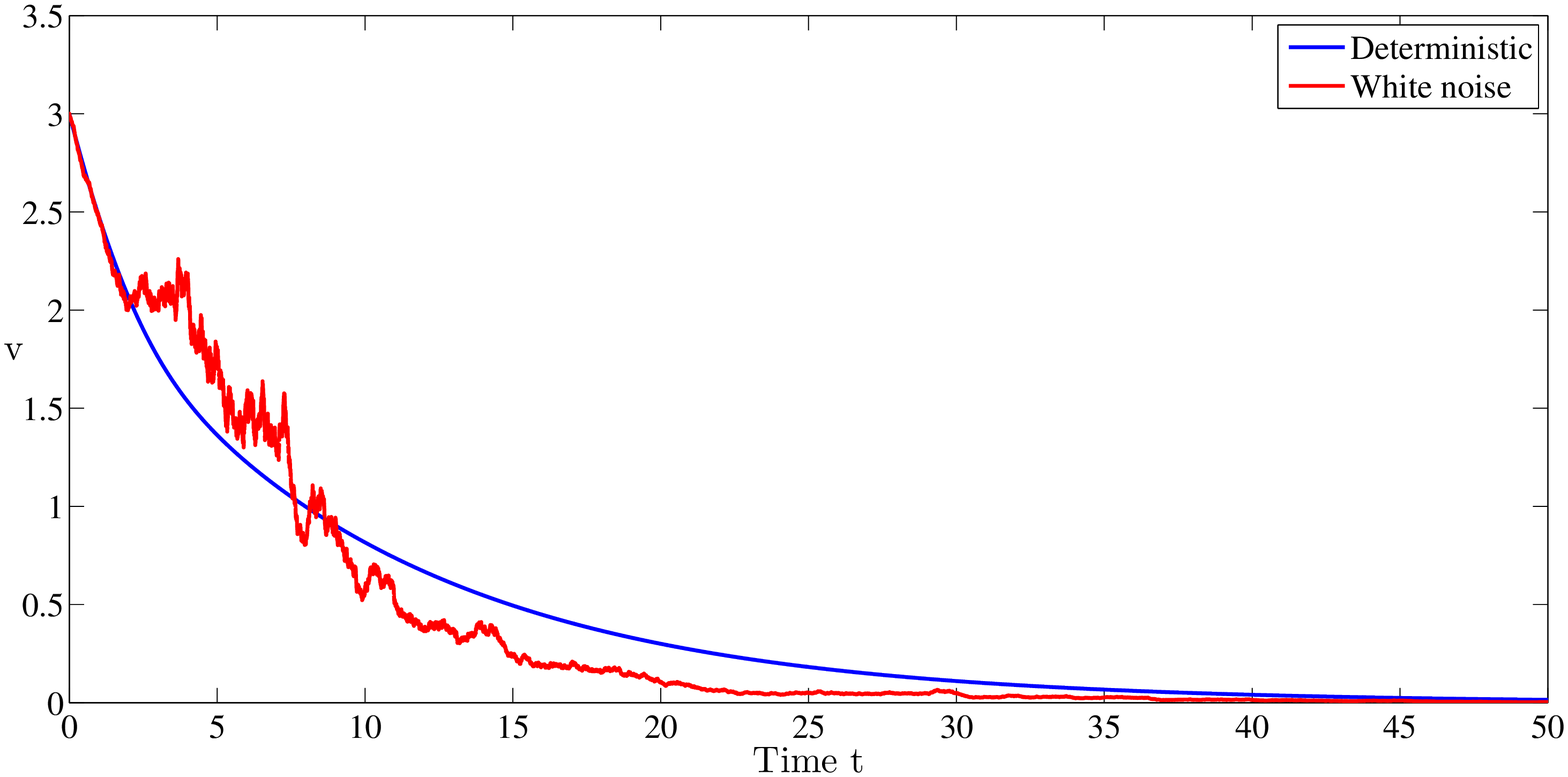}}
\caption{Extinction of the total labour force.}\label{fig1}
\end{figure}

The evolution of the free jobs and the total labour force, for both
deterministic and stochastic models, is also illustrated in Fig.~\ref{fig2}. 
In this case, the key conditions \eqref{eq:cond:pers} of our
Theorem~\ref{thm:4.1} are satisfied:
$R^s_0=\dfrac{r}{d}-\dfrac{\sigma^2K^2}{2d}=4.99>1$ 
and $m-\dfrac{r}{K}=0.09>0$. As predicted by Theorem~\ref{thm:4.1}, one can 
clearly see in Fig.~\ref{fig2} the persistence of the total labour force.
\begin{figure}
\centering
\subfloat[Behaviour of free jobs]{\label{fig:2a}
\includegraphics[scale=0.23]{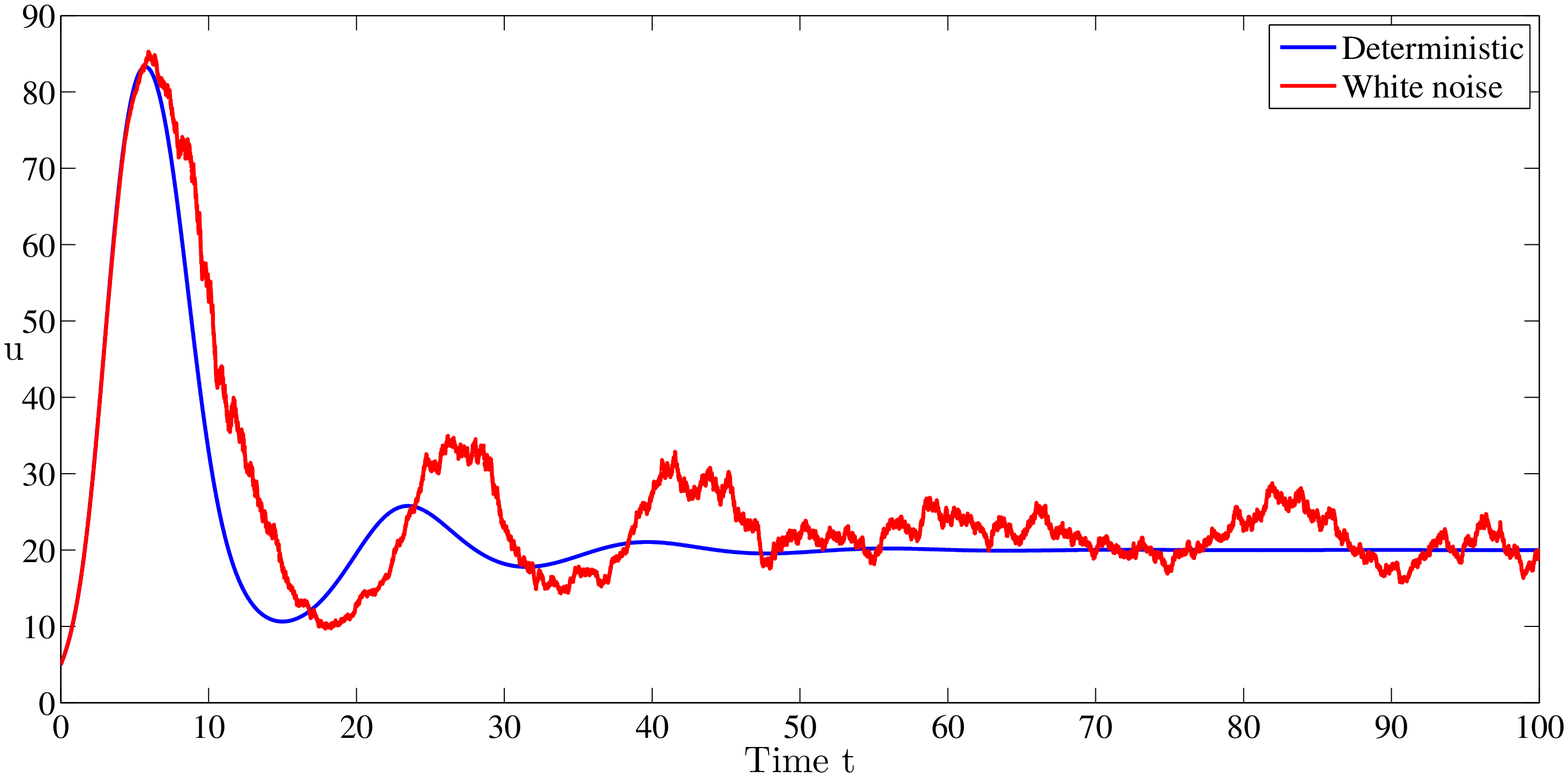}}\\
\subfloat[Behaviour of total labour force]{\label{fig:2b}
\includegraphics[scale=0.23]{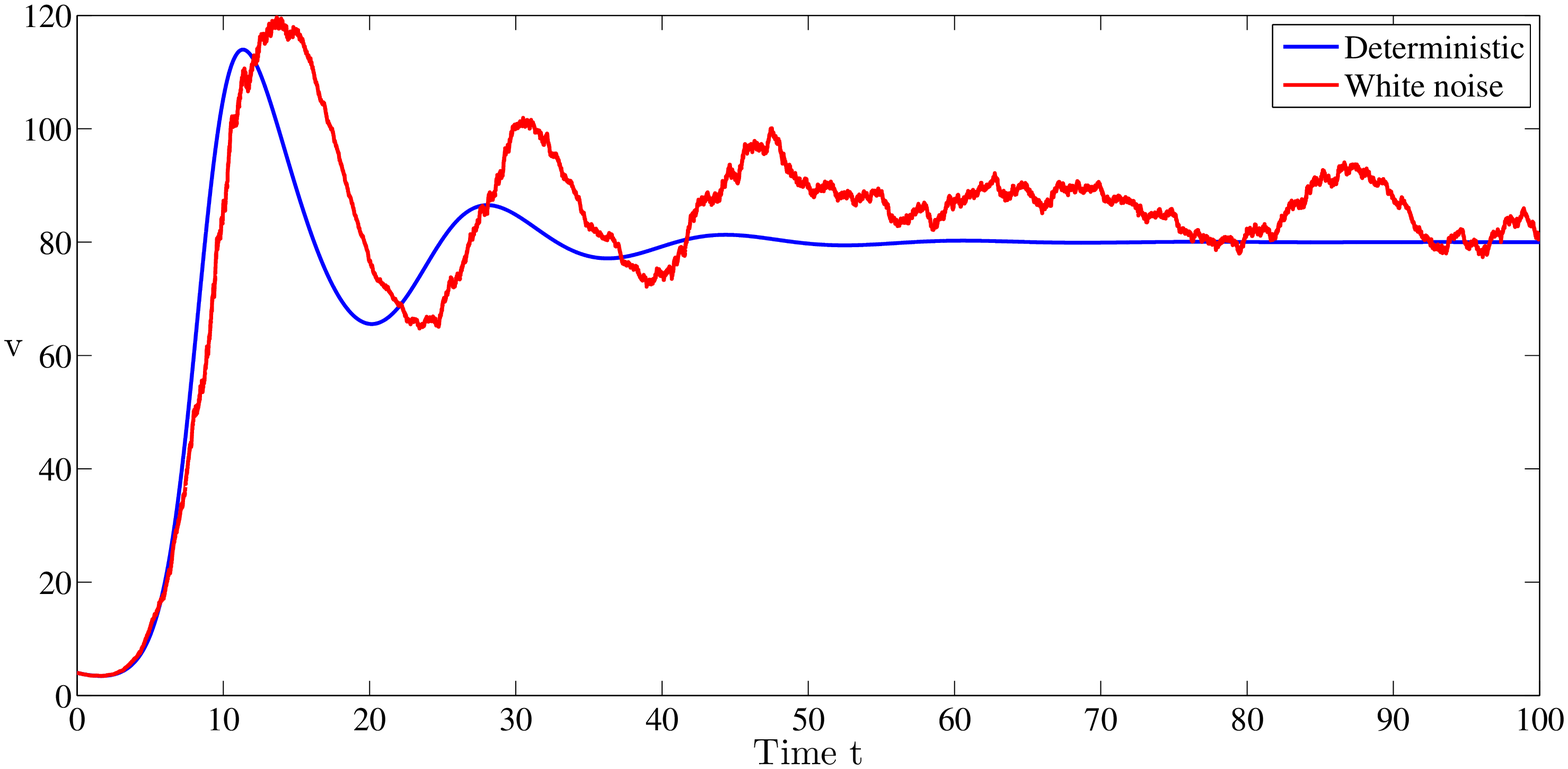}}
\caption{Persistence of the total labour force.}\label{fig2}
\end{figure}


\section{Conclusions}
\label{sec:6}

The labour force (workforce), which can be defined as the total number of people 
who are eligible to work, is a centred component of each modern economy, whereas 
free jobs are systematically supplied by companies.
In this work, we have proposed and analysed a capital-labour model 
by means of an economic dynamical system describing the interaction 
between free jobs and labour force. Mathematically, our model is 
governed by stochastic differential equations, where 
the component of stochastic noise is considered 
for an additional degree of realism, intended 
to describe well reality. Furthermore, the transmission rate by which 
the labour force individuals are moving to the free jobs compartment 
is modelled by the logistic growth function with an appropriate carrying capacity $K$. 
Some relevant results were obtained. First of all, by proving existence 
and uniqueness of a global positive solution, as well as its boundedness,
we have shown that the proposed model is mathematically and economically well-posed. 
Moreover, a sufficient condition for the extinction 
of labour force is obtained, via the strong law of large numbers for martingales, 
in addition to adequate sufficient conditions for the persistence in mean.
In order to illustrate our theoretical results, we have implemented some numerical 
simulations where, for a good accuracy of the approximate numerical solutions, 
the Milstein scheme has been used.


\subsection*{Acknowledgements}

This research is part of first author's Ph.D. project, 
which is carried out at University of Aveiro under 
the Doctoral Program in Applied Mathematics
of Universities of Minho, Aveiro, and Porto (MAP-PDMA).
It was partially supported by the 
Portuguese Foundation for Science and Technology (FCT) 
within project UIDB/04106/2020 (CIDMA).




\begin{thebibliography}{xx}
	
\bibitem{MR4154640}
T. Becker, 
A decomposition heuristic for rotational workforce scheduling, 
J. Sched. {\bf 23} (2020), no.~5, 539--554. 

\bibitem{9}
N. Dalal, D. Greenhalgh\ and\ X. Mao, 
A stochastic model of AIDS and condom use, 
J. Math. Anal. Appl. {\bf 325} (2007), no.~1, 36--53. 

\bibitem{MR4068867}
S. G\"{o}ttlich\ and\ S. Knapp, 
Uncertainty quantification with risk measures in production planning, 
J. Math. Ind. {\bf 10} (2020), Paper No.~5, 21~pp. 

\bibitem{11}
A. Gray, D. Greenhalgh, L. Hu, X. Mao\ and\ J. Pan, 
A stochastic differential equation SIS epidemic model, 
SIAM J. Appl. Math. {\bf 71} (2011), no.~3, 876--902. 

\bibitem{Higham}
D. J. Higham, 
An algorithmic introduction to numerical simulation 
of stochastic differential equations, 
SIAM Rev. {\bf 43} (2001), no.~3, 525--546. 

\bibitem{Lotfi:2019}
E. M. Lotfi, K. Hattaf\ and\ N. Yousfi, 
Stability and Hopf bifurcation of an epidemic model with logistic growth and delay,
Discontin. Nonlinearity Complex. {\bf 8} (2019), no.~4, 379--389.

\bibitem{Mahrouf:Axioms}
M. Mahrouf, A. Boukhouima, H. Zine, E. M. Lotfi, D. F. M. Torres and N. Yousfi, 
Modeling and forecasting of COVID-19 spreading by delayed stochastic differential equations,
Axioms {\bf 10} (2021), no.~1, Art.~18, 16~pp.
{\tt arXiv:2102.04260}

\bibitem{14}
X. Mao, 
{\it Stochastic differential equations and applications}, 
second edition, Horwood Publishing Limited, Chichester, 2008. 

\bibitem{Riad}
D. Riad,  K. Hattaf\ and\ N. Yousfi,
Dynamics of capital-labour model with Hattaf-Yousfi functional response,
J. Adv. Math. Comp. Sci. {\bf 18} (2016), no.~5, 1--7.

\bibitem{MR3891564}
T. Winberry, 
A method for solving and estimating heterogeneous agent macro models, 
Quant. Econ. {\bf 9} (2018), no.~3, 1123--1151. 

\bibitem{MR4173153}
H. Zine, A. Boukhouima, E. M. Lotfi, M. Mahrouf, D. F. M. Torres\ and\ N. Yousfi,
A stochastic time-delayed model for the effectiveness of Moroccan COVID-19 deconfinement strategy, 
Math. Model. Nat. Phenom. {\bf 15} (2020), Paper no.~50, 14~pp. 
{\tt arXiv:2010.16265}

\bibitem{MyID:456}
H. Zine\ and\ D. F. M. Torres,
A stochastic fractional calculus with applications to variational principles,
Fractal Fract. {\bf 4} (2020), no.~3, Art.~38, 11~pp.
{\tt arXiv:2008.00233}

\end{thebibliography}
\end{document}